\documentclass[12pt, a4paper]{article}
\usepackage{amsmath}
\usepackage{amsfonts}
\usepackage{amssymb}
\usepackage{amsthm}

\title{Representations of the quantum torus and applications to finitely 
presented groups}

\author{C. J. B. Brookes and J. R. J. Groves}

\numberwithin{equation}{section}

\newtheorem{defin}{Definition}
\newtheorem{lemma}{Lemma}[section]
\newtheorem{thm}[lemma]{Theorem}

\newtheorem*{thm*}{Theorem}
\newtheorem*{lemma*}{Lemma}
\newtheorem{prop}[lemma]{Proposition}
\newtheorem{coroll}[lemma]{Corollary}

\DeclareMathOperator{\Ann}{Ann}
\DeclareMathOperator{\rank}{rk}
\DeclareMathOperator{\Hom}{Hom}
\DeclareMathOperator{\Supp}{Supp}

\DeclareMathOperator{\im}{im}
\DeclareMathOperator{\kdim}{kdim}

\newcommand{\ore}{\"Ore }
\newcommand{\hol}{strongly holonomic}
\def\mbl{11cm}
\newcommand{\comment}[1]{}
\begin{document}
\maketitle

\section{Introduction}

The first strand of this paper concerns the (algebraic) quantum torus. 
By this we mean the crossed product $FA$ of a field $F$ with
a free abelian group $A$ of finite rank, and much of the time we shall assume 
that $F$ is the centre of $FA$. 

The representation theory of such crossed products depends heavily on 
the image in the multiplicative group of $F$ of the 2-cocycle used to define
the crossed product. For example in \cite{Brcp} the first author showed that 
the global dimension is equal to the maximal rank of a subgroup $B$ for 
which the sub-crossed product $FB$ is commutative. 
It is natural to concentrate on impervious $FA$-modules, those non-zero 
modules that contain no non-zero submodule induced from a module 
over a sub-crossed product $FB$ for a subgroup $B$ of infinite 
index in $A$.
(A consequence of this condition when $F$ is the centre of $FA$ is 
that $M$ is $FB$-torsion-free for every subgroup $B$ of $A$ for which 
the sub-crossed product $FB$ is commutative.)

In \cite{Brcp} it is shown that, under the assumption that the centre of $FA$
is exactly $F$, the Gelfand-Kirillov dimension, $\dim M$, of any
impervious $FA$-module $M$ is at least one half of the rank of $A$. 
This is analogous to Bernstein's inequality for Weyl algebras. 
Following the terminology used for Weyl algebras, we shall say that 
an $FA$-module $M$ is {\it strongly holonomic} if it is impervious 
and $\dim M$ is exactly one half of the rank of $A$. 
We prove the following:
\begin{thm*} [Theorem 4.2] Let $M$ be a \hol\  $FA$-module. Then
there is a subgroup of finite index in $A$ having the form $$
A_1\oplus\dots\oplus A_t$$ so that if $i\neq j$ then $FA_i$ commutes
with $FA_j$ and the 2-cocycles defining the crossed products $FA_i$
have infinite cyclic image in the multiplicative group of F, and these 
images of 2-cocycles are non-commensurable if $i\neq j$.

Further, if we consider
$M$ as $FA_i$-module then there exist  $FA_i$-submodules which are
\hol.
\end{thm*}
In fact one can relate $M$ to the tensor product of these $FA_i$-submodules. 
Roughly speaking, strongly holonomic modules 
are built from taking tensor products of strongly holonomic modules 
for crossed products defined using a 2-cocycle with infinite cyclic
image.

The second strand of the paper is to consider the structure of certain 
finitely presented groups. Twenty five years ago it seemed likely that 
an understanding of finitely presented abelian-by-nilpotent groups 
would quickly follow the understanding of finitely presented metabelian 
groups gained in 
Baumslag \cite{Bau73} and in  Bieri and Strebel \cite{BSmeta}.

It soon became clear, however, that even the existence of finitely 
presented abelian-by-nilpotent groups which were not also 
nilpotent-by-abelian-by-finite was a non-trivial question. It was 
settled, however, by Robinson and Strebel in \cite{RobStre}. They 
provided examples where the `nilpotent top'  was either a Heisenberg 
group of Hirsch length 3 or a direct product of such a group with an 
infinite cyclic group. It is not difficult to extend their techniques 
to provide examples with `nilpotent top' which are Heisenberg of any 
rank or, more generally, a central product of such groups and cyclic 
groups. (Here a Heisenberg group is one of the form
$$\langle x_1, 
\dots, x_n, y_1, \dots, y_n, z: [x_i,x_j]=[y_i,y_j]=[x_i, 
z]=[y_i,z]=1, [x_i,y_i]=z\rangle$$
where the indices $i,j$ are 
allowed to run between $1$ and $n$; and by `nilpotent top' we mean 
the quotient by the Fitting subgroup.)

In \cite{BRW}, the first 
author, Roseblade and Wilson showed that a finitely presented 
abelian-by-polycyclic group is virtually nilpotent-by-nilpotent. In 
 \cite{Brcp}, the first author showed that a finitely presented 
abelian-by-nilpotent group is virtually nilpotent-by-nilpotent of 
class 2. Thus the question  of finitely presented 
abelian-by-polycyclic groups essentially reduces to the case that the 
`polycyclic top' is nilpotent of class 2.

We argue that the restriction is, in some sense, greater still. 
In fact, with a natural restriction, we are reduced to 
examples which have a similar structure to the generalisations of the 
Robinson-Strebel examples mentioned above. 
We cannot in this case 
restrict the Fitting quotient in general because the second author and Strebel \cite{GS} have shown that
the set of possible Fitting quotients is closed under subdirect 
product and hence includes all finitely generated nilpotent groups of 
class 2. However we do make progress if we focus on subdirectly 
irreducible finitely presented groups, that is those where any two non-trivial 
normal subgroups have non-trivial intersection.

\begin{thm*} [Corollary 5.5] Let $G$ be a finitely presented group 
which is an extension of an abelian group by a group which is torsion-free 
nilpotent of class two. Suppose that $G$ is subdirectly irreducible. 
Then the quotient by the Fitting subgroup of $G$ has a subgroup of 
finite index which is a central product of groups which are either 
Heisenberg or cyclic.
\end{thm*}

It seems likely that the same will be true if one weakens `finitely 
presented' by replacing it with the condition that $G$ is the quotient 
of a small finitely presented group, that is one without free subgroups 
of rank two).

After some definitions and basic results in Section 2, we consider 
some properties of the geometric invariant, introduced in 
\cite{jrjgcjbb1}, for modules over crossed products. 
This invariant is related to one used by Bieri and Strebel in 
the classification of finitely presented metabelian groups \cite{BSmeta}, 
which in turn was related to the logarithmic limit set of Bergman \cite{Berg}. 
Such invariants are also of interest in tropical geometry \cite{Ein}.
In Section 3 we consider local cones and provide 
an alternative proof to Theorem B of Wadsley \cite{Wad} linking 
them with trailing coefficient moduIes. 
In Section 4, we prove the main result about strongly holonomic modules
and complete the paper by applying this result to groups in Section 5.

\section{Definitions and preliminary results}
Throughout this section and the next, $D$ will denote a division 
ring, $A$ a finitely
generated free abelian group and $DA$ a crossed product of $D$ with 
$A$. By the {\it rank} of
any abelian group $B$, we shall always mean the torsion-free rank or, 
equivalently, the
$\mathbb Q$-dimension of the tensor product $B \otimes \mathbb Q$ of 
$B$ with the rational numbers
$\mathbb Q$. We shall denote the rank of $B$ by $\rank B$. The rank 
of $A$ will always be
denoted by $n$. All modules will be right modules. All functions will 
be written on the left.

The structure of $DA$ demands that there is a $D$-vector space basis
$\bar A$ of $DA$, consisting of units and in bijective correspondence 
$a\rightarrow \bar a$
with $A$, and it is convenient to assume throughout that $\bar 1$ is 
the multiplicative
identity of $DA$; thus each element $\alpha$ of $DA$ can be uniquely 
expressed as a sum of the
form $$\alpha=\sum \bar a d_a $$ with $d_a\in D$, $\bar a \in \bar A$ 
and only finitely many
$d_a$ non-zero. We shall refer to the finite set of elements $a\in A$ 
such that $d_a$ is non-zero as
the {\it support} of $\alpha$, written $\Supp(\alpha)$, and for a 
subset $X$ of $A$ write $DX$
for the set of elements of $DA$ with support in $X$. 
The multiplication in $DA$ depends on a 2-cocycle with image in the 
multiplicative group of $D$.

If $B$ is a
submonoid of $A$ then $DB$ is a subring which has a natural structure
as a crossed product of $D$ with $B$. Because $A$ is torsion-free 
abelian of finite
rank, it is orderable and so it is easy to prove that $DA$ has no 
non-zero divisors of
zero. Further details, as well as proofs of some of the statements 
made here, can be found in
the book by Passman \cite {Passcr}.

We denote the homomorphism group $\Hom(A, \mathbb R)$ by $A^*$ and 
use similar notation for
other abelian groups. As $A$ has finite rank, $\chi(A)$ will also 
have finite rank; we call
this the {\it rank} of $\chi$. We often extend the definition of 
$\chi$ to $DA$ by
defining, for non-zero $\alpha\in DA$, $\chi(\alpha)=\min\{\chi(a)\}$ 
where $a$ is allowed to
run through the support of $\alpha$. If $B$ is a subgroup of $A$ then 
there is a natural map
from $A^*$ to $B^*$ obtained by restriction. We denote this map by 
$\pi_B$. If $C$ is a
subgroup of $B$ then the corresponding map from $B^*$ to $C^*$ is 
denoted by $\pi^B_C$.

  We shall denote by $M$ a finitely generated $DA$-module. Then $M$ 
defines a subset
$\Delta(M)$ of $A^*$. We refer to \cite{jrjgcjbb2} for a full 
definition but the most useful
characterisation for the current purposes is the following.
\begin{equation}\parbox{\mbl}{$\chi\notin \Delta(M)$ if and only if, 
for each $m\in M$, there
is a relation $m.(1+\alpha)=0$ with $\alpha\in DA$ and $\chi(a)>0$ 
for each $a$ in the support
of $\alpha$. }\end{equation}
We refer to \cite[Section 3]{jrjgcjbb2} for a fuller discussion of 
the elementary properties of
$\Delta(M)$.

The {\it dimension} of $M$ is the largest natural number $m$ so that 
$M$ contains a non-zero torsion-free $DB$-submodule for
some
$B\le A$ with $B$ of rank $m$. The properties of this dimension are 
discussed in \cite{jrjgcjbb2}; in particular, it is shown that it
coincides with the standard Gelfand-Kirillov dimension.

It turns out that it is much easier to describe a large subset of 
$\Delta(M)$. A point $x$ of $\Delta(M)$ is {\it regular}
if some neighbourhood of $x$ in $\Delta(M)$ is an $m$-ball for some 
positive integer $m$ and if $m$ is the largest integer
for which this can occur. Then $\Delta^*(M)$ is the Euclidean closure 
of the set of regular points of $\Delta(M)$. The main
result of \cite{jrjgcjbb2} was that, if $M$ has dimension $m$, then 
$\Delta^*(M)$ is a rational polyhedron and that the points lying in
$\Delta(M)$ but not in $\Delta^*(M)$ can be enclosed within a 
rational polyhedron of dimension $m-1$. (Here a {\it rational 
polyhedron}  
is a finite union of finite intersections of 
half-spaces with boundaries defined by a linear equation with 
rational coefficients. It is of 
{\it dimension $m$} if it contains 
$m$-balls but no $k$-balls for $k>m$.) 
More recently Wadsley \cite{Wad}
has shown that $\Delta(M)$ is itself polyhedral

 The local cone was introduced in \cite{jrjg1984}, where $DA$ is commutative,  in an attempt to describe local behaviour of
$\Delta$. Let $S\subseteq A^*$ and let $x\in S$. The {\it local cone} 
of $S$ at $x$ is
$$LC_{x}(S)=\{y: \text{ for some } \epsilon_0>0, x+\epsilon y\in S 
\text{ for all } \epsilon \in [0,\epsilon_0]\}.$$
Observe that the local cone is a cone, centered at the origin.  In 
all cases here,
$S$ will be either $\Delta(M)$ or $\Delta^*(M)$. The dimension of 
$LC_{\chi}(\Delta(M))$, for a regular point $\chi$, equals
the dimension of $\Delta(M)$ and so that of $M$.

In this and the next section, we shall be interested in the relation 
between the concept of local cone and the following
concept, which we can also regard as being `local'.

\begin{defin}  Let $M$ be a finitely generated $DA$-module furnished with a
finite generating set
$\mathcal X$. Fix $\chi\in A^*$ and set $A(0)=\{a\in A:\chi(a) \ge 0\}$ and
$A(+)=\{a\in A:\chi(a) >
0\}$. Then $A(0)$ and $A(+)$ are subsemigroups of $A$ and we can form the
sub-crossed products $DA(0)$
and
$DA(+)$. Define the {\rm trailing coefficient} module $TC_{\chi}(M)$ to be
$$TC_{\chi}(M)=\mathcal X.DA(0)/\mathcal X.DA(+);$$
it is naturally a module for $DB$ where $B$ is the kernel of $\chi$.
\end{defin}

Observe that, using the characterisation of $\Delta(M)$ above, it 
follows immediately that if $\chi \notin \Delta(M)$
then $TC_{\chi}(M)$ is zero. The converse follows from Proposition 
3.1 of \cite{jrjgcjbb2}. Because $TC_{\chi}(M)$ is a $DB$-module it
again has a $\Delta$-set, which is a subset of $B^*$.

Theorem B of \cite{Wad} establishes the relationship between the local cone 
at $\chi$ and $\Delta(TC_{\chi}(M))$. The main aim of the rest of this 
section and the next is to provide an alternative approach to that result.

First we wish to establish a useful technical condition for inclusion in 
$\Delta(TC_{\chi}(M))$. We have extracted  part of the proof of this 
as a
technical lemma. It enables us to apply results which are standard 
for Noetherian rings to non-Noetherian subrings of $DA$.
\begin{lemma}\label{ringlem} Let $U$ and $V_1$ be subsemigroups of 
$A$ and $V$ a submonoid of $A$ with $UV\subseteq U,
VV_1\subseteq V_1$ and $V_1\subseteq V$. Then
\begin{enumerate}
\item $R=DU+DV$ is a subring of $DA$ and $J=DU+DV_1$ is an ideal of $R$;
\item $1-J$ is a right denominator set in $R$ (in the sense of 
\cite[2.1.13]{MR});
\item if each $x\in \mathcal X$ is $(1-J)$-torsion then so also is 
each $m\in M$.
\end{enumerate}
\end{lemma}
\begin{proof} The first statement of the lemma is a routine check.

Set $T=1-J$. We show that $T$ is a right \ore set in $R$. Recall that 
this means that we must
show that, if $r\in R$ and $t\in T$ then there exist elements $r'\in 
R$ and $t'\in T$ so
that $rt'=tr'$.

The union of the supports of $r$ and $t$ is finite and so we can find 
finitely generated subsemigroups $\widehat U$ of
$U$, $\widehat V_1$ of $V_1$ and a finitely generated  submonoid 
$\widehat V$ of $V$ so that $r\in D\widehat
U+D\widehat V$ and
$t\in D\widehat U+D\widehat V_1$. Set $\tilde U=\widehat U\widehat V, 
\tilde V=\widehat V$ and $\tilde
V_1=\widehat V
\widehat V_1$. Then $\tilde U, \tilde V_1$ are subsemigroups and 
$\tilde V$ is a submonoid; further,
$\tilde U\tilde V\subseteq \tilde U$, $\tilde V\tilde V_1 \subseteq 
\tilde V_1$ and $\tilde V_1
\subseteq\tilde V$. Also
$$r\in \tilde R=D\tilde U+D\tilde V \text{ and } t\in \tilde J= 
D\tilde U+D\tilde V_1.$$
As in part (1) of the lemma, $\tilde R$ is a ring with ideal $\tilde J$.

We can use a non-commutative version of the Hilbert basis theorem 
(see, for example, Theorem 10.2.6 of \cite{Passcr} to
show that $\tilde R$ is Noetherian. Also,
$\tilde J$ is generated as ideal of $\tilde R$ by elements of $\bar 
A$ and if $a\in A$ then $\bar a \tilde
R=\tilde R \bar a$. Thus we can apply Proposition 2.6 of \cite{MR} 
and then Proposition 4.2.9 of \cite{MR} to show that
$1-\tilde J$ is a right \ore set. Thus we can find $r'\in \tilde R$ 
and $t'\in 1-\tilde J$ so that $rt'=tr'$.
As $r'\in R$ and $t'\in 1-J$, this shows, therefore, that $T$ is a 
right \ore set. Because $DA$ has no divisors of
zero, neither does $R$ and so $T$ is a right denominator set.

To prove the last part of the lemma, let $m\in M$ and suppose that 
$m=\sum xd\overline a$ with $x\in \mathcal X$,
$d\in D$ and $a\in A$. Suppose that, for $x\in \mathcal X$, we have 
$xt_x=0$ with $t_x\in T$. Then $(xd\overline a).
(t_x)^{d\overline a} =0$ and $(t_x)^{d\overline a}$ is still an 
element of $J$. It is a standard check, using the right \ore 
condition, that
a sum of
$T$-torsion elements is still $T$-torsion and so $m$ is also $T$-torsion.
\end{proof}

\begin{lemma}\label{techlem} Let $\chi\in A^*$ and let $B$ denote the 
kernel of $\chi$. Let $\psi\in B^*$. Then $\psi\notin 
\Delta(TC_{\chi}(M))$
if and only if
\begin{equation}\label{notintc}\parbox{\mbl}{
for each $m\in M$,
there exist $\alpha \in DA$ and $\beta\in DB$ with $m=m\alpha+m\beta$ 
and $\chi(\alpha)>0$ and $\psi(\beta)>0$.}
\end{equation}
\end{lemma}
\begin{proof}
Let $R$ denote $DA(+) + DB(0)$ and let $J$ denote $DA(+) + DB(+)$ of 
$R$. Applying  Lemma \ref{ringlem} with
$U=A(+), V=B(0)$ and $V_1=B(+)$, we see that $R$ is a subring of $DA$ 
and $J$ is an ideal of $R$. Further $1-J$ is a
right denominator set  in $R$.

Applying the definition of the Delta sets, $\psi\notin 
\Delta(TC_{\chi}(M))$ if and only if for each $u\in TC_{\chi}(M)$
there exists $\beta\in DB(+)$ with $u.(1+\beta)=0$. Applying the 
definition of the trailing coefficient module, this implies
\begin{equation}\parbox{\mbl}{\label{tempcond}
for each $x\in \mathcal X$,
there exist $\beta \in DB(+)$, $x_i\in \mathcal X$ and $\rho_i\in 
DA(+)$ so that $x.(1+\beta)=\sum_i x_i \rho_i$.}
\end{equation}
Reversing this last argument, we see that   $\psi\notin 
\Delta(TC_{\chi}(M))$ if and only if  (\ref{tempcond}) holds.

Let $N$ denote the $R$-submodule of $M$ generated by $\mathcal X$. If 
(\ref{tempcond}) holds then for each $x\in
\mathcal X$, we have
$x\in NJ$ from which it follows easily that
$N=NJ$ and reversing the argument shows that $N=NJ$ is equivalent to 
(\ref{tempcond}). Thus $\psi\notin
\Delta(TC_{\chi}(M))$ if and only if $N=NJ$.

Observe that (\ref{notintc}) is equivalent to the condition that $M$ 
is $(1-J)$-torsion and from (3) of Lemma
\ref{ringlem} this is equivalent to the condition that $N$ is 
$(1-J)$-torsion. Thus the lemma is reduced to showing that
$N$ is $(1-J)$-torsion if and only if $N=NJ$.

If $N$ is $(1-J)$-torsion, then clearly $N=NJ$. For the converse 
observe that, since $T=1-J$ is a right denominator set in $R$ we can
form the ring of quotients $R_T$ and the module of quotients $N_T$. 
Then $N_T$ is a finitely generated $R_T$ module satisfying
$N_T=N_TJ_T$. But $J_T$ lies in the Jacobson radical of $R_T$ and so, 
by Nakayama's lemma (see, for example 0.3.10 of \cite{MR}),
$N_T=0$. That is $N$ is $T$-torsion, as required.
\end{proof}
Observe that the proof of the lemma shows that it is sufficient, in 
(\ref{notintc}), to assume that the condition holds for all $m$
belonging to some generating set of $M$.

\begin{lemma}\label{ext}Suppose that $L\rightarrow M \rightarrow N$ 
is a short exact sequence of $DA$-modules. Then
$$\Delta(TC_{\chi}(M))=\Delta(TC_{\chi}(L))\cup \Delta(TC_{\chi}(N)).$$
\end{lemma}
\begin{proof} This is an immediate application of Lemma 
\ref{techlem}. \end{proof}

Much of the content of the next two sections will be to relate the 
local cone to the Delta set of the trailing coefficient
module. We begin with a relatively simple observation.

\begin{lemma}\label{firstLCtoTC}  Let $\chi \in A^*$ and let $B$ 
denote the kernel of $\chi$. Then
$$LC_{\Delta(M)}(\chi) \subseteq \pi_B^{-1}(\Delta(TC_{\chi}(M))).$$
\end{lemma}
\begin{proof}
Suppose that $\psi\in A^*$ with 
$\psi|_B=\pi_B(\psi)\notin\Delta(TC_{\chi}(M)$. We must show that
$\psi \notin LC_{\Delta(M)}(\chi)$.

Since $\psi|_B\notin \Delta(TC_{\chi}(M)$ then for each $x\in 
\mathcal X$ there exists $\beta_x \in DB$ with
$\left(x+\mathcal X.DA(+)\right).\beta_x=\{0\}$ and 
$\beta_x=1+\gamma_x$ with $\psi(\gamma_x)>0$. This implies that
$$x.(1+\gamma_x) = \sum_{y\in \mathcal X} y\alpha_{x,y}$$
with $\alpha_{x,y}\in DA$ and $\chi(\alpha_{x,y})>0$.

Choose $\epsilon_0$ by
$$0<\epsilon_0<\min \frac{\chi(a)}{|\psi(a)|}$$
where $a$ is allowed to range through the support of all the elements 
$\alpha_{x,y}$ with $x,y\in \mathcal X$. If $0< \epsilon
\le \epsilon_0$ then, for $a$ in the support of some $\alpha_{x,y}$
$$\chi(a)+\epsilon\psi(a) >0 .$$
Also, for $b$ in the support of some $\gamma_x$,
$$(\chi+\epsilon\psi)(b) =0+\epsilon\psi(b)>0.$$

Thus, for each $x\in X$ we have an expression of the form
$$x= \sum_{y\in X} y.\delta_{x,y}$$
with $\delta_{x,x}=\alpha_{x,x}-\beta_1$ and, if
$x\neq y$, $\delta_{x,y}=\alpha_{x,y}$. Thus 
$(\chi+\epsilon\psi)(\delta_{x,y})>0$. It follows that
$\chi+\epsilon\psi\notin \Delta(M)$ for $0<\epsilon<\epsilon_0$. Thus 
$\psi\notin  LC_{\Delta(M)}(\chi)$, as required.

\end{proof}

\begin{lemma}\label{firstineq}  Let $\chi \in A^*$ and let $B$ denote 
the kernel of $\chi$. Then
$$\rank(\chi)+\dim_{DB}(TC_{\chi}(M))\le \dim_{DA} (M).$$
\end{lemma}
\begin{proof}
Pick an isolated subgroup $B_1$ of $B$ of rank equal to the dimension 
of $TC_{\chi}(M)$ so that, for some
$\overline y\in TC_{\chi}(M)$, we have $\overline y.DB_1\cong DB_1.$ 
Now choose an isolated subgroup $C$ of $A$
so that $B+C=A$ and so that $C\cap B=B_1$. Choose $y\in M$ so that 
$y$ has the image $\overline y$ in
$TC_{\chi}(M)$. We claim that $y.DC \cong DC$. If this claim is true, then
\begin{eqnarray*}
\dim(M) \ge 
\rank(C)&=&\rank(A/B)+\rank(B_1)\\&=&\rank(A/B)+\dim(TC_{\chi}(M))\\&=&\rank(\chi)+\dim(TC_{\chi}(M)).
\end{eqnarray*}
(Recall that $\rank(\chi)=\rank(\im(\chi))=\rank(A/\ker(\chi))$.)

If $y.DC$ is not isomorphic to $DC$ then there is some element 
$\alpha\in DC$ such that $y.\alpha=0$. By
multiplying $\alpha$, if necessary, by a suitable element of $C$, we 
can assume that $\alpha=\alpha_0+\alpha_1$ with
$\chi(a)=0$ for every $a$ in the support of $\alpha_0$ and 
$\chi(\alpha_1)>0$ and $\alpha_0\neq 0$. Thus $\alpha_0\in DB\cap
DC=DB_1$. Passing to
$TC_{\chi}(M)$, we have $\overline y.\alpha_0=0$. This contradicts 
the assumption that
$\overline y.DB_1\cong DB_1$ and so completes the proof of the claim.
\end{proof}

\begin{prop} \label{firsteq} If $\chi \in \Delta^*(M)$ then
$$\rank(\chi)+\dim_{DB}(TC_{\chi}(M))= \dim_{DA} (M).$$
\end{prop}
\begin{proof} Suppose that the dimension of $M$ is $m$. If $\chi\in 
\Delta^*(M)$
then $\chi$ lies in at least one polyhedron of dimension $m$ within 
$\Delta(M)$ and so $LC_{\Delta(M)}(\chi)$
has dimension $m$. By Lemma \ref{firstLCtoTC}, 
$\pi_B^{-1}(\Delta(TC_{\chi}(M)))$ has dimension at least $m$ and so
$TC_{\chi}(M)$ has dimension at least $m-r$ where $r$ is the 
dimension of the kernel of $\pi_B$. But $r$ is just
the rank of $\chi$. Thus
$$\dim(TC_{\chi}(M))\ge m-r=\dim M - \rank (\chi).$$
Combining this with Lemma \ref{firstineq} gives the result.
\end{proof}

\begin{lemma}\label{secondLCtoTC}  Let $\chi \in \Delta^*(M)$ and let 
$B$ denote the kernel of $\chi$. Then
$$LC_{\Delta^*(M)}(\chi) \subseteq 
\pi_B^{-1}(\Delta^*(TC_{\chi}(M))).$$\end{lemma}
\begin{proof}
Suppose that $\Delta(M)$ has dimension $m$. Then $\Delta^*(M)$ is a 
finite union of convex polyhedra of dimension $m$. Thus
$LC_{\Delta^*(M)}(\chi)$ has dimension $m$ and so 
$\pi_B(LC_{\Delta^*(M)}(\chi))$ has dimension at least $m-r$ where $r$
is the dimension of the kernel of $\pi_B$ or, equivalently, the rank 
of $\chi$. As $\pi_B(LC_{\Delta^*(M)}(\chi))\subseteq
\pi_B(LC_{\Delta(M)}(\chi))$, then, from Lemma \ref{firstLCtoTC}, 
$\pi_B(LC_{\Delta^*(M)}(\chi))$ is a subset of\\
$\Delta(TC_{\chi}(M))$ having dimension at least $m-r$. But 
Proposition \ref{firsteq} tells us that the dimension of
$\Delta(TC_{\chi}(M))$ is exactly $m-r$. Thus 
$\pi_B(LC_{\Delta^*(M)}(\chi))$ is actually a subset of
$\Delta^*(TC_{\chi}(M))$, as required.
\end{proof}

\section{Trailing coefficient modules and local cones}
We retain the notation of the previous section. In particular, $D$ is 
a division ring, $A$ is an abelian group, $DA$ is a crossed
product of $D$ by $A$, $\chi\in A^*$, $B$ is the kernel of $\chi$ and 
$M$ is a $DA$-module. The aim in this section is to prove
equality in Lemma \ref{secondLCtoTC}. We begin with a simple case in 
Lemma \ref{onerelLCtoTC} and then  proceed to a less restricted case, 
the
`co-dimension one'  case, in Lemma \ref{codimoneLCtoTC}. Then, in 
Proposition \ref{LCtoTC}, we use the fact that an $m$-dimensional 
Delta-set can
be reconstructed from its projections onto $m+1$-dimensional 
subspaces to reduce the general case to the `co-dimension one' case.

\begin{lemma}\label{onerelLCtoTC}  Suppose that $M$ is a cyclic 
1-relator module. Let $\chi \in \Delta^*(M)$ and
let $B$ denote the kernel of $\chi$. Then
$$LC_{\Delta^*(M)}(\chi) = \pi_B^{-1}(\Delta^*(TC_{\chi}(M))).$$
\end{lemma}
\begin{proof}
If $M$ is a cyclic 1-relator module with relator $r$, then 
$\Delta(M)$ is described in Proposition
2.3 of \cite{jrjgcjbb1}. For each $\chi\in A^*$, write 
$r=r_{\chi}+s_{\chi}$ where if $a, b$ are in the support of
$r_{\chi}$ and if $c$ is in the support of $s_{\chi}$ then 
$\chi(a)=\chi(b)<\chi(c)$. By multiplying $r$ by the inverse of some
element of the support of $r_{\chi}$, we can and will assume that 
$\chi(a)=0$ for all $a$ in the support of $\chi$.

Then $\chi\in\Delta(M)$ if and only if the support of $r_{\chi}$ 
contains more than one element. In this latter case, it is
easily verified that $TC_{\chi}(M)$ (using the same generator as was 
used for $M$) is a 1-relator module with relator
$r_{\chi}\in DB$. Thus $\Delta(TC_{\chi}(M))$ is calculated in an 
analogous way to that used for $\Delta(M)$; that is,
$\psi\in \Delta(TC_{\chi}(M))$ if and only if $(r_{\chi})_{\psi}$ has 
support with more than one element.  It is an easy
consequence of the definition, or the description in 
\cite{jrjgcjbb1}, that $\Delta(M)=\Delta^*(M)$
for one-relator modules $M$; a similar comment then holds for $TC_{\chi}(M)$.

Suppose that
$\psi\in \Delta(TC_{\chi}(M))$ and choose $\phi\in A^*$ so that 
$\pi_B(\phi)=\phi|_B=\psi$. Because $\chi(s_{\chi})>0$,
we can choose $\epsilon_0$ so that
$\chi(c)+\epsilon\phi(c)>0$ for $c$ in the support of $s_{\chi}$ and 
$0<\epsilon\le \epsilon_0$. Consider
$\chi+\epsilon\phi$ for $0< \epsilon\le \epsilon_0$. We have that
\begin{eqnarray*}
r_{\chi+\epsilon\phi} &=& (r_{\chi})_{\chi+\epsilon\phi}\quad \text{ as
$(\chi+\epsilon\phi)(s_{\chi})>0$}\\
&=&(r_{\chi})_{\epsilon\phi}\\
&=&(r_{\chi})_{\psi} \quad \text{ as $r_{\chi} \in DB$.}
\end{eqnarray*}
By assumption, $(r_{\chi})_{\psi}$ has support with more than one 
element and hence so also does $r_{\chi+\epsilon\phi}$.
Thus $\chi+\epsilon\phi\in \Delta(M)$ for $0<\epsilon\le \epsilon_0$ and so
$\phi\in LC_{\Delta(M)}(\chi)$.

We have thus shown that
$\pi_B^{-1}(\Delta^*(TC_{\chi}(M))) \subseteq 
LC_{\Delta^*(M)}(\chi)$. The reverse inclusion is provided by
Lemma \ref{secondLCtoTC} and so the proof is complete.
\end{proof}

\begin{lemma}\label{inducedTC}  Suppose that $A_1$ is an isolated 
subgroup of $A$ and that $N$
is a $DA_1$-module. Let $\chi \in A^*$ and let $B$ denote the kernel 
of $\chi$. Let
$\chi_1=\pi_{A_1}(\chi)$. Then
$$\Delta(TC_{\chi}(N\otimes_{DA_1} DA))= (\pi^B_{B\cap A_1})^{-1} 
\Delta(TC_{\chi_1}(N))$$ and so
$$\Delta^*(TC_{\chi}(N\otimes_{DA_1} DA))= (\pi^B_{B\cap A_1})^{-1} 
\Delta^*(TC_{\chi_1}(N)).$$
\end{lemma}

\begin{proof} Let $\psi\in B^*$ and suppose that $\psi\notin 
\Delta(TC_{\chi}(N\otimes_{DA_1} DA))$.
By Lemma \ref{techlem}, for each $n\in N$, there exists $\alpha_n
\in DA$ and $\beta_n \in DB$ with $n\otimes 1=(n\otimes 1)(\alpha_n 
+\beta_n)$ and $\chi(\alpha_n)>0$,
$\psi(\beta_n)>0$. Fix a transversal
$\mathcal T$, containing 1, for $A_1$ in $A$; we can do this so that 
it contains a transversal for $B\cap A_1$ in $B$. Each
element of the support of $\alpha_n$ can be written uniquely as a 
product of an element in $A_1$ and an element in $\mathcal
T$. Thus
$\alpha_n$ can be written as
$$\alpha_n =\sum_{t\in \mathcal T} \alpha_n(t)t\qquad\text{ with 
}\alpha_n(t) \in DA_1.$$
Clearly $\chi_1(\alpha_n(1))>0$. Similar comments apply for $\beta_n$;
in particular $\beta_n(1)\in D(B\cap A_1)$ and $\psi(\beta_n(1))>0$. 
Because $N\otimes_{DA_1} DA$ is an induced module, it
follows that
$n=n.(\alpha_n(1)+\beta_n(1))$. Thus the restriction of $\psi$ to 
$A_1$, that is $\pi^B_{B\cap A_1}(\psi)$, does
not lie in  $\Delta(TC_{\chi_1}(N))$.

Suppose, conversely, that $\pi^B_{B\cap A_1}(\psi)$ does not lie in 
$\Delta(TC_{\chi_1}(N))$.
Then, by Lemma \ref{techlem}, for each $n\in N$, there exist 
$\alpha_n\in DA_1$ and $\beta_n \in
D(B\cap A_1)$ such that $n=n.(\alpha_n+\beta_n)$ and $\chi_1(a)>0$ 
for each $a$ in the support of $\alpha_n$ and
$\psi_1(b)>0$ for each $b$ in the support of $\beta_n$.  Clearly, 
$\alpha_n\in DA$ and  $\beta_n\in DB$ with $\chi(\alpha_n)>0$
and $\psi(\beta_n)>0$. We thus have condition (\ref{notintc}) holding 
for those $m\in N\otimes_{DA_1} DA$ of the form
$n\otimes 1$. But the latter elements suffice to generate 
$N\otimes_{DA_1} DA$ as $DA$-module and so, using the comment at the
end of the proof of Lemma \ref{techlem}, it follows that $\psi\notin 
\Delta(TC_{\chi}(N\otimes_{DA_1} DA))$.

  The final equality of the lemma is an immediate deduction.
\end{proof}

\begin{lemma}\label{codimoneLCtoTC}  Suppose that $A$ has rank $n$ 
and that $M$ is a $DA$-module of dimension $n-1$. Let
$\chi \in \Delta^*(M)$ and let $B$ denote the kernel of $\chi$. Then
$$LC_{\chi}(\Delta^*(M)) = \pi_B^{-1}(\Delta^*(TC_{\chi}(M))).$$
\end{lemma}

\begin{proof} Observe firstly that $M$ has a finite series with 
quotients $\{M_1,\dots, M_s\}$
which are cyclic and critical. Further, as $M$ has dimension $n-1$, 
all $M_i$ have dimension at
most $n-1$ and at least one has dimension exactly $n-1$. By Lemma  \ref{ext},
$$\Delta(TC_{\chi}(M))=\cup_{i=1}^s\Delta(TC_{\chi}(M_i)).$$
By Lemma \ref{firsteq}, $TC_{\chi}(M_i)$ has dimension $\dim 
M_i-\rank \chi$ and hence $\Delta^*(TC_{\chi}(M_i))$ also has this 
dimension. Thus
$\Delta^*(TC_{\chi}(M))$ is the union of those
$\Delta^*(TC_{\chi}(M_i))$ for which
$M_i$ has dimension $n-1$. Similarly, $LC_{\chi}(\Delta^*(M))$ is the 
union of those
$LC_{\chi}(\Delta^*(M_i))$ for which $M_i$ has dimension $n-1$. Thus 
we need prove the result
only in case $M$ is cyclic and critical. We shall use an inductive 
argument on the rank of $A$.

We deal firstly with the case that there is a subgroup $A_1$ of $A$ 
with $A/A_1$ infinite cyclic
so that $M$ is not torsion-free as $DA_1$-module. Then the set of 
$DA_1$-torsion elements of $M$
forms a non-zero $DA$-submodule $M_1$ of $M$. Let $N$ be a critical 
$DA_1$-submodule of $M_1$. By
Lemma 2.4 of \cite{jrjgcjbb2}, $N.DA$ has dimension at most $\dim 
N+1$ with equality if and only if $N.DA$ is
induced from $N$. Since $A_1$ has rank $n-1$ and $N$ is a torsion 
$DA_1$-module, $N$ has
dimension at most $n-2$. As $M$ is critical, every non-zero 
submodule, in particular $N.DA$, also has dimension $n-1$. It
follows that $N.DA$ is induced from $N$, that is $n.DA\cong
N\otimes_{DA_1} DA$, and that $N$ has dimension exactly $n-2$.

By Corollary 4.5 of \cite{jrjgcjbb2}, $\Delta^*(M)=\Delta^*(N.DA)$ 
and so, by Lemma 3.4 of \cite{jrjgcjbb2},
\begin{equation}\label{indeq}
\Delta^*(M)=\pi^{-1}_{A_1}(\Delta^*(N)).
\end{equation}
Set $\chi_1=\pi_{A_1}(\chi)=\chi|_{A_1}$. The
inductive argument enables us to assume that
\begin{equation}\label{indass}
LC_{\chi_1}(\Delta^*(N)) = (\pi^{A_1}_{B\cap 
A_1})^{-1}(\Delta^*(TC_{\chi_1}(N)))
\end{equation}
Observe also that the quotient $M/N.DA$ has smaller dimension than 
$M$, because $M$ is critical. By Lemma \ref{ext} we have
$$\Delta(TC_{\chi}(M))=\Delta(TC_{\chi}(N.DA))\cup 
\Delta(TC_{\chi}(M/N.DA))$$ and the dimensions show that
$$\Delta^*(TC_{\chi}(M))=\Delta^*(TC_{\chi}(N.DA).$$

We therefore have
\begin{eqnarray*}
\pi_B^{-1}(\Delta^*(TC_{\chi}(M))) &=& 
\pi_B^{-1}(\Delta^*(TC_{\chi}(N\otimes_{DA_1} DA)))\\
&=& \pi_B^{-1}((\pi^B_{B\cap A_1})^{-1}\Delta^*(TC_{\chi}(N)) \text{ 
by Lemma \ref{inducedTC} }\\
&=&\pi_{B\cap A_1}^{-1}(\Delta^*(TC_{\chi}(N))\\
&=&\pi_{A_1}^{-1} ((\pi^{A_1}_{B\cap A_1})^{-1}(\Delta^*(TC_{\chi}(N)))\\
&=&\pi_{A_1}^{-1} (LC_{\chi_1}(\Delta^*(N)) \text{ by (\ref{indass})}\\
&=&LC_{\chi} \pi_{A_1}^{-1}(\Delta^*(N))\\
&=&LC_{\chi}(\Delta^*(M)) \text{ by (\ref{indeq}).}\\
\end{eqnarray*}
This completes the proof in case $M$ is not torsion-free as $DA_1$-module.

Thus we can assume that $M$ is torsion-free as $DA_1$-module for each 
subgroup $A_1$ with
$A/A_1$ infinite cyclic. Because we are assuming that
$M$ is critical, it follows that every proper quotient of $M$ has 
dimension at most $n-2$ and so must
be torsion as $DA_1$-module for each subgroup $A_1$ with $A/A_1$ 
infinite cyclic. Thus we have the necessary conditions
for Theorem 2.4 of \cite{jrjgcjbb1}  and we can easily deduce from 
the proof of this theorem that
\begin{equation}\label{Vint}\Delta^*(M) = (\Delta(V_1) \cap 
\Delta(V_2))^*\end{equation} where $V_1$ and $V_2$ are 1-relator 
$DA$-modules
each of which has a quotient isomorphic to $M$. Using the fact that 
each module has a quotient isomorphic to $M$, together with   Lemma 
\ref{ext}, we
deduce that
\begin{equation}\label{ineq}\Delta(TC_{\chi}(M)) \subseteq 
\Delta(TC_{\chi}(V_1)) \cap \Delta(TC_{\chi}(V_2)).\end{equation}
Observe that (\ref{Vint}) together with the fact that each 
$\Delta(V_i)$ has dimension $n-1$, shows that $\chi\in 
\Delta^*(V_i)$. By Proposition
\ref{firsteq}, each of the three $\Delta$-sets in (\ref{ineq}) has 
the same dimension, equal to
$(n-1)-\rank(\chi)$. Thus we can replace (\ref{ineq}) by
\begin{equation}\label{MVV}
\Delta^*(TC_{\chi}(M)) \subseteq (\Delta(TC_{\chi}(V_1)) \cap 
\Delta(TC_{\chi}(V_2)))^*.
\end{equation}

Thus we have
\begin{eqnarray*}
LC_{\chi}(\Delta^*(M)) &=& LC_{\chi}((\Delta(V_1) \cap \Delta(V_2))^*)\\
&=&(LC_{\chi}(\Delta(V_1)) \cap LC_{\chi}(\Delta(V_2)))^*\text{ using 
the definition }\\
&\ &   \text{\hspace{7.3cm} of local cones}\\
&=&(\pi_B^{-1}(\Delta(TC_{\chi}(V_1))) \cap 
\pi_B^{-1}(\Delta(TC_{\chi}(V_2))))^*\text{ by Lemma 
\ref{onerelLCtoTC}}\\
&=&\pi_B^{-1}(\Delta(TC_{\chi}(V_1))\cap \Delta(TC_{\chi}(V_2)))^*\\
&\supseteq& \pi_B^{-1}(\Delta^*(TC_{\chi}(M)))\text{ by (\ref{MVV})}.
\end{eqnarray*}
The reverse inequality has been proved in Lemma \ref{secondLCtoTC} 
and so the proof of the lemma is
complete.
\end{proof}
\begin{prop}\label{LCtoTC}  Suppose that $M$ is a $DA$-module. Let
$\chi \in \Delta^*(M)$ and let $B$ denote the kernel of $\chi$. Then
$$LC_{\chi}(\Delta^*(M)) = \pi_B^{-1}(\Delta^*(TC_{\chi}(M))).$$
\end{prop}

\begin{proof} By Theorem 4.4 of \cite{jrjgcjbb2}, we know that 
$\Delta^*(TC_{\chi}(M))$ is a rational polyhedron and hence so also
is  $\pi_B^{-1}(\Delta^*(TC_{\chi}(M)))$. By Proposition 
\ref{firsteq}, we know that the latter has dimension $m$, say,
equal to that of $M$. Thus
$\pi_B^{-1}(\Delta^*(TC_{\chi}(M)))$ is a finite union of
$m$-dimensional convex polyhedra. The same holds true for 
$\Delta^*(M)$ and Theorem 4.4 of \cite{jrjgcjbb2} shows that the remainder of $\Delta(M)$ 
is also contained
in a finite union of $m$-dimensional convex polyhedra. Let 
$S=\pi_B^{-1}(\Delta^*(TC_{\chi}(M)))\cup \Delta(M)$.

We wish to apply Theorem 4.3 of \cite{jrjgcjbb2}. The requirement on 
$S$ stated in the theorem is satisfied because $S$ itself has been 
chosen to lie in a
polyhedron of dimension $m$. There therefore exists a finite set 
$\mathcal X$ of
$m$-dimensional subspaces of $A^*$ and a finite set of projections 
$\pi_i$ with image  $A_i^*$, with $A_i$ a subgroup of $A$ of rank 
$m+1$, so that, if we
denote by $M_i$ the set $M$ considered as $DA_i$-module, then
\begin{enumerate}
\item $S\subseteq \bigcup_{ X\in \mathcal X} X$;
\item for each $i$, $\ker(\pi_i)$ meets each element of $\mathcal X$ 
trivially (in the language of \cite{jrjgcjbb2}, this follows from the 
fact that $\pi_i$ is
regular with respect to $\mathcal X$);.
\item $\Delta^*(M)= \cap_i \pi_{A_i}^{-1}(\pi_{A_i}(\Delta^*(M)))$ 
(this follows from (3) of Theorem 4.3 of \cite{jrjgcjbb2} in the same 
way that (4) of
that theorem  follows from (3)).
\end{enumerate}
Using (1) and (2) above, together with Proposition 3.8 of 
\cite{jrjgcjbb2} ,  we see that $M_i$ is finitely generated for each 
$i$. By Proposition 3.7 of
\cite{jrjgcjbb2}, $\pi_{A_i}(\Delta(M))=\Delta(M_i)$. Using (2) 
above, $\pi_{A_i}$ maps each element of $\mathcal X$ faithfully and 
so $\dim \Delta(M)
=\dim \Delta(M_i)$. Thus
\begin{equation}\label{eqp0}\pi_{A_i}(\Delta^*(M))=\Delta^*(M_i).\end{equation}

Therefore we can replace (3) above by
\begin{equation}\label{eqp1}
  \Delta^*(M)= \cap_i \pi_{A_i}^{-1}(\Delta^*(M_i)).
\end{equation}

Observe that (\ref{eqp1}) implies easily that
\begin{equation} \label{eqp2}
LC_{\chi}(\Delta^*(M))= \cap_i \pi_{A_i}^{-1}(LC_{\chi_i}(\Delta^*(M_i)))
\end{equation}
where $\chi_i$ denotes the restriction of $\chi$ to $A_i$.
Note that the dimension of $M_i$ is one less than the rank of $A_i$. 
Thus we can apply Lemma \ref{codimoneLCtoTC} to show
that
\begin{equation} \label{eqp3}
LC_{\chi_i}(\Delta^*(M_i)) = (\pi^{A_i}_{B\cap 
A_i})^{-1}(\Delta^*(TC_{\chi_i}(M_i))).
\end{equation}
and so, combining (\ref{eqp2}) with (\ref{eqp3}), we have
\begin{equation} \label{eqp4}
LC_{\chi}(\Delta^*(M))= \cap_i (\pi_{B\cap 
A_i})^{-1}(\Delta^*(TC_{\chi_i}(M_i))).
\end{equation}

A straightforward application of  Lemma \ref{techlem} shows that
\begin{equation} \label{eqp5}
\pi^B_{B\cap A_i} ( \Delta(TC_\chi(M))) \subseteq \Delta (TC_{\chi_i}(M_i)).
\end{equation}

We wish to establish the version of (\ref{eqp5}) where the 
$\Delta$-sets are replaced by their corresponding
$\Delta^*$-versions. This is immediate once we know that the 
dimensions on each side of the equation coincide.
We have chosen the subgroups $A_i$ so that the subspaces 
$\ker(\pi_i)$ meet the elements of $\mathcal X$ trivially. In
particular, this implies that the $\ker(\pi_i)$ intersect the 
supporting spaces (that is the spaces spanned by the convex polyhedra 
making up the Delta set)
of $\pi_B^{-1}(\Delta^*(TC_{\chi}(M)))$ trivially. It follows easily 
that the kernel of $\pi^B_{B\cap
A_i}$ meets each supporting space of $\Delta^*(TC_{\chi}(M))$ 
trivially and that $\ker(\pi_{A_i})$ meets  $\ker(\pi_{B})$
trivially. Thus $\pi^B_{B\cap A_i} ( \Delta(TC_\chi(M)))$ has 
dimension equal to that of $\Delta(TC_\chi(M))$.

It follows from (\ref{eqp0}) that  $\chi_i \in \Delta^*(M_i)$. By 
Proposition \ref{firsteq}, the dimension of $\Delta(TC_\chi(M))$ is 
$m-\rank \chi$ and
the dimension of $\Delta(TC_{\chi_i}(M_i))$ is $\dim M_i - \rank 
\chi_i$. We have already seen that $\dim M =\dim M_i$. Since 
$\ker(\pi_i)$ meets
$\ker(\pi_{B}$) trivially, $A_i$ supplements $B$ in $A$. Recalling 
that $B$ is the kernel of $\chi$, it follows that $\rank \chi_i$ = 
$\rank \chi|_{A_i}$ =
$\rank \chi$. Thus the dimensions of the two sides of the inequality 
in (\ref{eqp5}) are equal and so we have
\begin{equation} \label{eqp6}
\pi^B_{B\cap A_i} ( \Delta^*(TC_\chi(M))) = (\pi^B_{B\cap A_i} ( 
\Delta^*(TC_\chi(M))))^*
\subseteq \Delta^* (TC_{\chi_i}(M_i)).
\end{equation}
Thus,
\begin{eqnarray*} \label{eqp7}
\pi_B^{-1}(\Delta^*(TC_{\chi}(M)) &\subseteq& \cap 
_i\pi_B^{-1}((\pi^B_{B\cap A_i})^{-1}(\Delta^*
(TC_{\chi_i}(M_i))) \qquad\text{ by (\ref{eqp6}).}\\
&=& \cap _i \pi_{B\cap A_i}^{-1}(\Delta^* (TC_{\chi_i}(M_i)))\\
&=& LC_{\chi}(\Delta^*(M))\qquad\text{ by (\ref{eqp4}).}
\end{eqnarray*}
The reverse inclusion is Lemma \ref{secondLCtoTC} and so the proof of 
the proposition is complete.
\end{proof}

Let $C$ be a subgroup of $A$ and let $V=\ker(\pi_C)$. A character 
$\chi\in A^*$ is {\it generic} for $V$ if
$\chi\in V$ and  $TC_{\chi}(M)$ is locally of finite dimension as a 
module for $DC$. More discussion of genericness can be found
in Section 3 of [BG3]

\begin{coroll} \label{geneqlc}$\chi$ is generic for $V$ if and only 
if $LC_{\chi}(\Delta^*(M))\subseteq V$.

\end{coroll}

\begin{proof}
By Proposition \ref{LCtoTC},
$$LC_{\chi}(\Delta^*(M))\subseteq V \text{ if and only if } 
\pi_B^{-1}(\Delta^*(TC_{\chi}(M))) \subseteq V.$$
But,
$\pi_B^{-1}(\Delta^*(TC_{\chi}(M))) \subseteq V $ if and only if 
$\ker (\pi_B) \subseteq V $ and
$\Delta^*(TC_{\chi}(M)) \subseteq \pi_B(V)$.
But $\ker (\pi_B) \subseteq V$ if and only if $\chi\in V$ and, by 
Lemma 2.4 of \cite{jrjgcjbb3}, $\Delta^*(TC_{\chi}(M)) \subseteq
\pi_B(V)$ if and only if  $\Delta^*(TC_{\chi}(M))$ is locally of 
finite dimension as a module for the dual of $\pi_B(V)$.
The latter condition is exactly that required for $\chi$ to be 
generic for $V$ and so the proof is complete.
\end{proof}
Observe that this implies immediately that any point of $\Delta^*(M)$ 
which is non-generic for some carrier space must lie
in at least two distinct carrier spaces. Proposition 3.6 of 
\cite{jrjgcjbb3}  guarantees that, under certain technical
conditions, a carrier space contains non-generic points. We can use 
that with Corollary \ref{geneqlc} to show that, with suitable easily 
satisfied assumptions,  
any two carrier spaces of a Delta-set must 
have non-zero intersection. 
(See, for example, Lemma \ref{nontrivinter}.)

\section{Modules over central crossed products}
\subsection{Preliminaries}
In this section, $F$ will denote a field, $A$ will denote a free 
abelian group of rank $n$ and $FA$ will be a crossed
product of $A$ with $F$ in which $F$ is central. Thus $FA$ is an 
$A$-graded algebra in which each component is isomorphic to $F$.

Recall from Theorem 4.4 of \cite{jrjgcjbb2} that if $M$ is a finitely generated 
$FA$-module of dimension $m$  then $\Delta^*(M)$ is a finite union of
$m$-dimensional convex polyhedra. The subspace spanned by one of 
these polyhedra is called a {\it carrier space} of $\Delta^*(M)$.
It is also shown in [BG2] that this subspace is rationally defined 
and so is the kernel of a projection map $\pi_B: A^* \rightarrow
B^*$ where $B$ is a subgroup, which we can assume isolated, of $A$. 
We call such a subgroup $B$ a {\it carrier space subgroup} of $M$.

Before stating the main result of this section, it will be convenient 
to collect a couple of results of the first author \cite{Brcp}. They
will be key tools in the following.
\begin{prop}[Brookes] \label{cjbb} Let $M$ be a non-zero finitely 
generated module over $FA$. Then
\begin{enumerate}
\item $\dim M= n-\rank B$ for some subgroup $B$ of $A$ with $FB$ commutative.
\item A carrier space subgroup of $M$ has a  subgroup $B$ of finite 
index with $FB$ commutative.
\item If $F$ is the centre of $FA$ and $M$ is impervious then $M$ is 
$FB$-torsion-free for every  subgroup $B$ of $A$ with $FB$ 
commutative.
\end{enumerate}
\end{prop}
\begin{proof} (1) is Theorem 3 of \cite{Brcp} and (3) is Theorem 4. Part (2) 
is an immediate consequence of the proof of Theorem 3.
Theorem 3 is proved by taking a character in $\Delta(M)$ of maximal 
rank $m$ and showing that some subgroup of finite index in the kernel 
of that character has a subgroup $B$ of finite index with $DB$ 
commutative. If we start with a carrier space subgroup then this 
corresponds
to a carrier space which must contain a character $\chi$ of maximal 
rank. Then it is easy to show that the kernel of $\chi$ is exactly 
the carrier space subgroup and the claim follows.
\end{proof}

In the following, $\zeta(A)$ is the largest subgroup of $A$ so that 
$F\zeta(A)$ is central in $FA$.
Similarly for any subgroup $B$ of $A$ 
we define $\zeta(B)$ to be the largest subgroup of $B$ for which 
$F\zeta(B)$ is central in $FB$.

\begin{defin} 
Suppose that $F$ is the 
centre of $FA$ and let $M$ be a finitely generated $FA$-module. We 
shall say that $M$ is {\rm \hol} if 
\begin{enumerate}
  \item $\rank 
A =2\dim M$;
  \item if $B$ is any subgroup of $A$ so that $FB$ is 
commutative, then $M$ is 
$FB$-torsion-free.
\end{enumerate}

\end{defin}

The modules that 
arise from the study of modules over nilpotent groups will have 
$\zeta(A)$ trivial and will be impervious. But then condition 2 is 
guaranteed by (3) of Proposition \ref{cjbb}.
Further, (1) and (3) of 
Proposition \ref{cjbb} guarantee that $2\dim M \ge \rank A$ and we 
will be able to guarantee, from the finite presentation of the groups 
from which they arise, that $2\dim M \le \rank A$. Thus we also have 
the first requirement for \hol.

This section is largely devoted to a proof of the following result:

\begin{thm} \label{main} Let $M$ be a \hol\  $FA$-module. Then there 
is a subgroup  of finite index in $A$ having the form 
$$ 
A_1\oplus\dots\oplus A_t$$ so that 
\begin{enumerate}
   \item if 
$i\neq j$ then $FA_i$ commutes with $FA_j$;
\item the centre of 
$FA_i$ is $F$;
\item the image of the 2-cocycle used to define each crossed product
$FA_i$ has infinite cyclic image in the multiplicative group of $F$. 
\item any two such images of 2-cocycles are not  commensurable.
\end{enumerate}
Further, if we consider $M$ as 
$FA_i$-module then there exist  $FA_i$-submodules which are \hol.
\end{thm}

We shall henceforth suppose that $FA$ has a \hol\  module $M$. We 
begin with some of the immediate consequences of this. We shall 
denote by $m$ the dimension of $M$.

\begin{lemma} \label{largecentre} Let $B$ be a subgroup of $A$. 
Suppose that $B$ contains a subgroup $C$
of rank $m$ with $FC$ commutative and suppose also that
$$\rank B+\rank \zeta(B)\ge 2m.$$
Then equality holds above. Also, if we denote by $F_1$ the field of fractions of $F\zeta(B)$  then there is a crossed product $F_1(B/\zeta(B))$,  a localisation of $FB$,  which has a \hol\  module. 

\end{lemma}
\begin{proof} Let $W$ denote a non-zero finitely generated 
$FB$-submodule of $M$. Clearly, $\dim_B(W)\le \dim_A M=m$. But
$C\le B$ and, by supposition, $M$ is $FC$-torsion-free. Thus $\dim_B 
W\ge m$ and so $\dim_B W=m$.

Because $F\zeta(B)$ is commutative, $M$ and so $W$, is torsion-free 
as $F\zeta(B)$-module. Set $W_1=W\otimes_{F\zeta(B)} F_1$. This is a 
module for $FB\otimes_{F\zeta(B)} F_1$ and it is straightforward to 
verify that $FB\otimes_{F\zeta(B)} F_1$ is isomorphic to a cross 
product of $F_1$ by $B/\zeta(B)$ which we shall denote by 
$F_1(B/\zeta(B))$. We claim that $W_1$ is \hol.

If $C_1$ is a subgroup of $B$ containing $\zeta(B)$ then 
$F_1(C_1/\zeta(B))$ is commutative if and only if $FC_1$ is also 
commutative. Also $W$ is torsion-free as $FC_1$-module if and only 
if  $W_1$ is torsion-free as $F_1(C_1/\zeta(B))$-module. 

Thus if 
$C_2$ is a subgroup of $B/\zeta(B)$ with $F_1C_2$ commutative then 
$W_1$ is $F_1C_2$-torsion-free and the second condition for 
`\hol' is satisfied by $W_1$.

It also follows that 
\begin{equation}
\label{one}
\dim 
W_1=m-\rank(\zeta(B)).
\end{equation}By (1) of Proposition 
\ref{cjbb}, there is a subgroup $C_2$ of $B$ with $F_1(C_2/\zeta(B))$ 
commutative and  
$$\dim W_1 
=\rank(B/\zeta(B))-\rank(C_2/\zeta(B)).$$
Thus $FC_2$ is commutative 
and so $\rank(C_2)\le m$. Thus 
\begin{equation}
\label{two}
\dim W_1 
=\rank(B)-\rank(C_2)\ge \rank(B)-m.
\end{equation}
Thus, combining 
(\ref{one}) and (\ref{two}), we have 
\begin{equation}
\label{three}
\rank(B)-m\le 
m-\rank(\zeta(B)).
\end{equation}
It follows that 
$\rank(B)+\rank(\zeta(B)\le 2m$. But we have assumed the reverse 
inequality and so the inequalities in (\ref{two}) and (\ref{three}) 
are, in fact, equalities. In particular, 
$$\dim(W_1)=\rank(B)-m= 
(1/2)(\rank(B)-\rank(\zeta(B)))=(1/2)\rank(B/\zeta(B)).$$
Thus $W_1$ 
is a \hol\ module.
\end{proof}

\begin{lemma}\label{trivinter}Let $C$ be a  subgroup of $A$ with $FC$ 
commutative. Then $C$ meets some carrier space subgroup of $M$ trivially.
\end{lemma}
\begin{proof}
By assumption $M$ is torsion-free over $FC$. Thus any non-zero 
submodule, $M_1$ say, of $M$ will also be torsion-free over
$FC$ and so will again have dimension $m$. In particular, we can take 
$M_1$ to be critical.

By Theorem 5.5 of [BG2], $(\pi_C(\Delta^*(M_1)))^*=\Delta^*(N)$ for 
some cyclic critical $FC$-submodule $N$ of
minimal dimension in $M$. Because $M$ is $FC$-torsion-free, $N$ must 
have dimension $\rank C$; that
is, $(\pi_C(\Delta^*(M_1)))^* =C^*$. Thus $\pi_C(\Delta^*(M_1)) = 
C^*$. But then, because a Euclidean space cannot be the union of 
finitely many proper subspaces, at least one of the
$m$-dimensional convex polyhedra contained in $\Delta^*(M_1)$ must 
map onto $C^*$. That is, $\ker \pi_C+V=A^*$ for some carrier
space $V$. Because $V$ is a rational subspace of $A^*$, we have 
$V=\ker \pi_B$ for some (carrier space) subgroup $B$ of $A$ and so
$C \cap B= \{1\}$.
\end{proof}

\begin{lemma}\label{nontrivinter}If $\rank A\ge 4$, then each carrier 
space of $M$ has non-trivial intersection with some
other carrier space.
\end{lemma}

\begin{proof} Let $V$ be a carrier space of $M$ and suppose that $C$ 
is the corresponding carrier space subgroup. By
Lemma \ref{cjbb}, $FC_1$ is commutative for some subgroup $C_1$ of 
finite index in $C$. Thus $M$ is torsion-free over $FC_1$; also the 
co-dimension of $V$ is $m$ which is at least 2. Thus, by Corollary 
3.7 of [BG3], $V$ contains non-zero points which are non-generic for 
$V$. It follows from Corollary \ref{geneqlc} that there
are points $\chi\in V$ such that $LC_{\chi}(\Delta^*(M))\not\subseteq 
V$. But then $\chi$ must lie in some other carrier
space of $V$ and so $V$ intersects some other carrier space, as required.
\end{proof}

\subsection{Alternating bilinear maps on vector spaces}
Our aim is 
to investigate the nature of the group generated by $\bar{A}$ within 
$DA$. In particular, we have considerable amounts of information 
about the possible subgroups $B$ of $A$ for which $DB$ is 
commutative. The structure is described by the commutator map from 
$A$ to $F$. But rather than work with $A$ and its subgroups, we shall 
work with the divisible hull $A\otimes \mathbb Q$ and its subspaces. 
The next paragraph translates the previous definitions to this 
context.

 Let $V$ and $W$ be finite dimensional vector spaces over a 
field $K$ and let $\phi: V \times V \rightarrow W$ be an alternating 
bilinear map. We shall often abbreviate $\phi(x,y)$ by $(x,y)$. We 
shall use terminology with group-theoretical overtones rather than 
than that derived from the theory of forms on vector spaces. We shall 
say that, if $x,y\in V$ then $x$ {\it centralises} $y$ if $(x,y)=0$ 
and if $S \subseteq V$ then the centraliser $\mathcal C(S)$ is the 
subspace of all those elements of $V$ which centralise each element 
of $S$. The centraliser of $V$ is called the {\it centre} of $V$. We 
shall say that a subspace $U$ of $V$ is {\it abelian} if any two 
elements of $U$ centralise each other.

We now turn to the sort of 
structure we wish to establish for crossed products with a \hol\ 
module.
\begin{defin} A {\rm symplectic base} for $V$ is a decomposition of 
$V$,
$$V=\bigoplus_{i=0}^t V_i,$$
as a direct sum so 
that
\begin{enumerate}
  \item $V_0$ is the centre of $V$;
  \item if 
$i\neq j$ then $(V_i,V_j)=\{0\}$;
  \item $(V_i, V_i)$ has dimension 
1 and $V_i$ has centre $\{0\}$;
  \item if $i\neq j$ then $(V_i, 
V_i)\neq (V_j,V_j)$.
\end{enumerate}
\end{defin} 

Observe that 
$\phi$ restricted to $V_i$ for $i>0$ yields a non-degenerate 
symplectic form on $V_i$ and so the well-known properties of such a 
form hold. In particular, every non-zero element of $V_i$ has 
centraliser, in $V_i$, of co-dimension one. Also, if $A$ is an 
abelian subspace of maximal dimension in $V_i$ with basis 
$\{x_1,\dots,x_m\}$ then $V_i$ has a basis $\{x_1,\dots,x_m, y_1, 
\dots, y_m\}$ where $(y_i,y_j)=0$ for each $i,j$ and $(x_i,y_j)=0$ 
precisely if $i\neq j$.

The definition of a symplectic base is 
designed to ensure a degree of uniqueness.

\begin{lemma} If $V$ has 
a symplectic base (as above) then the subspaces $V_0+V_i$ are 
uniquely determined up to re-arrangement. In particular, if the 
centre is trivial then the subspaces $V_i$ are unique up to 
re-arrangement.
\end{lemma}
\begin{proof} This follows immediately 
from the fact that $V_0$ is the centre of $V$ and that `the non-zero 
elements of $\cup_{i>0} (V_0+V_i)$ are precisely the elements of $V$ 
with centraliser of co-dimension at most one'. We prove the latter 
statement. Observe firstly that the elements of $V_0+V_i$ certainly 
have co-dimension at most 1 since elements of $V_0$ have centraliser 
$V$ and elements of $V_i$ with $i\neq 0$ have centraliser in $V_i$ of 
co-dimension 1.

If $v\in V$ then the function $\phi_v: V \rightarrow 
W$ given by $w \to \phi(v,w)$ is a linear map with kernel the 
centraliser of $v$. Thus the co-dimension of the centraliser of $v$ 
is equal to the dimension of the image of $\phi_v$. Let 
$v=\sum_{i=0}^t v_i$ with $v_i \in V_i$ and suppose, for example, 
that $v_1, v_2\neq 0$. Then there exist $v_1'\in V_1$ and $v_2'\in 
V_2$ so that $(v_i, v_i')$ is non-zero for $i=1,2$. Thus $(v,v_1')$ 
and $(v,v_2')$ are non-zero elements of $(V_1,V_1)$ and $(V_2,V_2)$ 
respectively and so are independent. Thus at most one of the $v_i 
(i\ge 1)$ can be non-zero if the centraliser of $v$ has co-dimension 
at most one. The proof is 
complete.
\end{proof}

\begin{prop}\label{absubsp} Suppose that $V$ 
has  a symplectic base $V=\bigoplus_i V_i$ with trivial centre and 
dimension $2m$. Let $U$ be an abelian subspace of dimension  at least 
$m$. Then 
$$U=\bigoplus_{i=1}^t(U\cap V_i).$$
Thus $U$ has 
dimension exactly $m$ and $U\cap V_i$ has dimension one half of the 
dimension of $V_i$.
\end{prop}
Observe that it is easy to deduce a 
similar statement without the assumption of trivial centre by 
applying the proposition to $V/V_0$.

\begin{proof} We prove by 
induction on the dimension of $V$.

\noindent {\it Case 1: suppose that 
$U$ contains a non-zero element of some $V_i$. }

Let us suppose that 
$0\neq v \in V_1\cap U$. Then there exists some $v'\in V_1$ so that 
$(v,v')\neq 0$. Set $X=\langle v,v' \rangle$. Then we can find a 
complement $X_1$ of $X$ in $V_1$ so that $(X,X_1)=\{0\}$. It is 
easily checked that $X_1, V_2, \dots, V_t$ forms a symplectic base 
for the sum $V'=X_1+ V_2+ \dots+ V_t$. Further, $V=X\oplus V'$; let 
$\pi$ be the projection of $V$ onto $X$. 

We claim that $\pi(U)$ has 
dimension 1. Otherwise, $\pi(U)=X$ and so $v'=\pi(u)$ for some $u\in 
U$. But then $(v,u)=0$ as both $v$ and $u$ are in the abelian 
subspace $U$. However, $(v,u)=(v,\pi(u))$ as $v\in X$ and 
$u-\pi(u)\in V'$. Also $(v, \pi(u))=(v,v')\neq 0$, a 
contradiction.

Thus $\pi(U)$ has dimension 1 and so $U_1=U\cap 
\ker(\pi)$ has dimension at least $m-1$. 
Since $V'=\ker(\pi)$ has 
dimension $2m-2$, we can apply the inductive hypothesis to show that 

$$U_1=(U_1\cap X_1) \oplus (U_1\cap V_2)\oplus\dots\oplus (U_1\cap 
V_t).$$
Also, as $v\notin \ker(\pi)$, $U=\langle v, U_1\rangle$. Thus 

\begin{eqnarray*}U&=& \langle v \rangle\oplus U_1\\&=& \langle v 
\rangle\oplus (U_1\cap X_1) \oplus (U_1\cap V_2)\oplus\dots\oplus 
(U_1\cap V_t)\\
&=&(U\cap V_1) \oplus \dots (U\cap 
V_t).\end{eqnarray*}
Thus the proof is complete in this 
case.

\noindent {\it Case 2: suppose that $\pi_i(U)$ is a proper subspace 
of $V_i$ where  $\pi_i$ is the projection onto $V_i$.}

Let us suppose 
that $\pi_1(U)\neq V_1$. Then $\pi_1(U)$ has a non-zero centraliser 
in $V_1$; say $0\neq v\in V_1$ centralises every element of 
$\pi_1(U)$. But then $v$ centralises every element of $U$ and 
$\langle v, U\rangle$ is still abelian. But then Case 1 applied to 
$\langle v, U\rangle$ completes the proof.

\noindent {\it Case 3: suppose 
that $U$ contains an element $u$ of the form $v_i+v_j$ where $v_i$ 
and $v_j$ are non-zero elements of $V_i$ and $V_j$ with $i\neq j$.}

Using 
Case 2, we can suppose that $\pi_i(U)=V_i$ and so there is an element 
$u'\in U$ with $(v_i, \pi_i(u'))\neq 0$. But $(u,u')=0$ as both $u$ and 
$u'$ are non-zero elements of the abelian subspace $U$ and so 

$$0=(u,u')=(v_i+v_j, u')=(v_i,\pi_i(u'))+(v_j,\pi_j(u')).$$
Thus we 
also have that $(v_j,\pi_j(u'))$ is non-zero and that the 
(one-dimensional) subspaces $(V_i,V_i)$ and $(V_j, V_j)$ are equal. 
But this was prohibited in the definition of a symplectic 
base.

\noindent {\it Case 4: the general case.}
By Case 2, we can assume that 
$\pi_1(U)=V_1$. Let $\dim(V_1)=2m_1$ and let $U_1=U\cap \ker(\pi_1)$. 
Thus $\dim(U_1)=\dim(U)-2m_1\ge m-2m_1$.

Choose an abelian subspace 
$X_1$ of dimension $m_1$ in $V_1$ and set $X=U\cap \pi_1^{-1}(X_1)$. 
Then $\dim X=\dim U-m_1\ge m-m_1$. 

Observe that, if $u\in U$ and 
$u=\pi_1(u)$ then $u\in V_1$ and so we can use Case 1 if $u\neq 0$. 
Thus the map $id-\pi_1$ is injective on $U$ and so $Y=(id-\pi_1)(X)$ 
is a subspace of $\ker(\pi_1)$ having dimension at least $m-m_1$.

We claim that $Y$ is abelian. For $i=1,2$ let $x_i\in X$ so that 
$y_i=x_i-\pi_1(x_i)\in Y$.
Then
\begin{eqnarray*}
(y_1,y_2) & = & 
(y_1,x_2)-(y_1, \pi_1(x_2)) =(y_1,x_2)\ \hbox{ as $y_1\in 
\ker(\pi_1)$ and $\pi_1(x_2)\in V_1$}\\
 & = 
&(x_1,x_2)-(\pi_1(x_1),x_2) =-(\pi_1(x_1),x_2)\ \hbox{ as $x_1,x_2\in 
U$}\\
& = &-(\pi_1(x_1),x_2)+(\pi_1(x_1), \pi_1(x_2))\ \hbox{ as 
$\pi_1(x_i)\in X_1$}\\
& = &-(\pi_1(x_1), x_2-\pi_1(x_2))=0\  \hbox{ 
as above.}
\end{eqnarray*}

Thus $Y$ is an abelian subspace of 
$\ker(\pi_1)$ and we can apply the inductive hypothesis to show that 
$Y$ is the direct sum of the $Y\cap V_i$. In particular, $Y\cap V_2$ 
is non-zero; say $v\in V_2$ and $v\neq 0$ with $v\in Y$. So 
$v=x-\pi_1(x)$ with $x\in X$. But then $x\in U$ and $x=\pi_1(x)+v$ 
with $\pi_1(x)\in V_1$ and $v\in V_2$. This falls under one of Cases 
1 or 3 and so completes the proof.
\end{proof}

We now turn to 
establishing the existence of a symplectic base from the existence of 
`sufficient' large abelian subspaces. The latter will arise in the 
application because of the existence and size of the carrier space 
subgroups corresponding to {\hol} modules.

Suppose that $\dim V+\dim 
\zeta(V)$ is even; equal, say, to $2m$. 

\begin{defin}We shall say 
that $V$\ {\rm has ample abelian subspaces} if whenever $X$ is a 
subspace of $V$ containing an abelian subspace of dimension $m$ and satisfying
$\dim X+\dim \zeta(X)\ge 2m$ then equality holds and there exists a 
non-empty set $\Omega_X$ of $m$-dimensional abelian subspaces of $X$ 
so that:
\begin{enumerate}
  \item if $\dim(X/\zeta(X))>2$ then, 
given $U_1\in \Omega_X$, there exists $U_2\in \Omega_X$ with $U_1\cap 
U_2 > \zeta(X)$;
  \item given any abelian subspace $U$ of $X$, there 
exists $U_1\in \Omega_X$ such that $U\cap U_1 \le 
\zeta(X)$.
\end{enumerate}
\end{defin}
Observe that the subspaces 
$X$, which contain abelian subspaces of dimension $m$ and satisfy the 
inequality, inherit the property of having ample abelian 
subspaces.

\begin{prop} \label{ample}If $V$ has ample abelian 
subspaces then $V$ has a symplectic base.
\end{prop}

\begin{proof} 
Observe that, if $V$ has ample abelian subspaces, then so also does 
$V/\zeta(V)$ and if $V/\zeta(V)$ has a symplectic base then so also 
does $V$. Thus we can assume that the centre of $V$ is zero. We shall 
use induction on the dimension of $V$.

Let $U$ be any abelian 
subspace of $V$ with dimension $m$. Then there exists $U_1\in 
\Omega_V$ so that $U\cap U_1=\{0\}$. There also exists $U_2\in 
\Omega_V$ so that $U_1\cap U_2$ is not zero; set $k=\dim(U_1\cap 
U_2)$. Then, as the $U_i$ are abelian of dimension $m$, we have that 
$X=U_1+U_2$ has dimension $2m-k$ and centre at least $U_1\cap U_2$. 
Thus $\dim X+\dim \zeta(X)\ge 2m$. Thus  equality holds and  $X$ has 
ample abelian  subspaces. Using the inductive hypothesis, $X$ 
therefore has a symplectic base. Also, as 
$$\dim X+\dim \zeta(X)= 
2m=\dim X+\dim(U_1\cap U_2),$$
it follows that $U_1\cap 
U_2=\zeta(X)$.

Consider $U\cap X$. As $U_1\le X$ and $U\cap 
U_1=\{0\}$, with both $U$ and $U_1$ of dimension half that of $V$, it 
follows that $U+U_1=V$ and so $U+X=V$. Thus $U\cap X$ has  dimension 
$m-k$. Further, as $\zeta(X)=U_1\cap U_2$, we have $(U\cap X)\cap 
\zeta(X)=\{0\}$. Choose a complement $X'$ to the centre of $X$ which 
contains $U\cap X$. Then $X'$ has a symplectic base $X'=\oplus X_i$ 
with trivial centre and, by Proposition \ref{absubsp}, 
$$U\cap 
X=\bigoplus_i (U\cap X)\cap X_i.$$

Form a new abelian subspace 
$U_3'$ of $X'$ by taking abelian subspaces in $X_i$ which, for $i>1$ 
complement $ (U\cap X)\cap X_i$ and for $i=1$, have dimension equal 
to that of $ (U\cap X)\cap X_1$ but intersect it in dimension 1. The 
existence of such subspaces within the $X_i$ follows easily from the 
fact that the restriction of $\phi$ to $X_i$ is a non-degenerate 
form.  Let $U_3=U_3'+\zeta(X)$.  Then $U_3$ has dimension $m$ and 
$U_3\cap U=U_3\cap(U\cap X)$ has dimension 1. Let $X''=U+U_3$. As 
before, we can show that $X''$ has ample abelian subspaces and so the 
inductive hypothesis tells us that $X''$ has a symplectic base with 
centre of dimension 1. Thus we have shown that every abelian subspace 
of dimension $m$ lies in a subspace of dimension $2m-1$ with centre 
of dimension 1 and a symplectic base.

Let $V_1$ and $V_2$ be two 
such subspaces of co-dimension 1 with centres $Z_1$ and $Z_2$ of 
dimension 1. Suppose that $(Z_1,Z_2)\neq \{0\}$. Set $V_3=V_1\cap 
V_2$ and $Z_3=Z_1+Z_2$. Then $(V_3,Z_3)=\{0\}$ and $V_3\cap 
Z_3=\{0\}$ so that $V=V_3\oplus Z_3$. But $V_2=V_3\oplus Z_2$ and so 
$V_3$ has a symplectic base. Hence so also does $V$.

We are left 
with the possibility that, whenever $V_1$ and $V_2$ are subspaces of 
co-dimension 1 with a symplectic base then their centres commute. 
That is, the subspace $Y$ spanned by all centres of such subspaces of 
co-dimension one is abelian. But then there is a subspace $U\in 
\Omega_V$ so that $Y\cap U=\{0\}$. We have shown, however, that $U$ 
can be placed inside a subspace $V'$ of co-dimension 1 with a 
symplectic base and that $U$ must then contain the centre of $V'$ and 
so intersect $Y$. Thus this case is not possible and the proof is 
complete.
\end{proof}
\subsection{Proof of Theorem \ref{main}}

Let 
$S$ be a free generating set for $A$ and let $\bar S=\{\bar a: a\in S\}$ 
be its image in $FA$. Let $N$ denote the multiplicative subgroup of 
$FA$ generated by $\bar S$. Then the derived subgroup $N'$ of $N$ 
will lie in $F$ and $A$ will be isomorphic to $N/N'$.

Thus 
commutation in $N$ will yield an alternating ($\mathbb{Z}$-)bilinear 
map \\$\hat \phi: A \times A \longrightarrow N'$. Set $V=A \otimes 
\mathbb{Q}$ and $W=N' \otimes \mathbb{Q}$. Then $\hat \phi$ extends 
to an alternating $\mathbb{Q}$-bilinear map $\phi:V \times V 
\longrightarrow W$. In the language of the previous section, we wish 
to show that $V$ has a symplectic base. Thus we will need to show 
that $V$ has ample abelian subspaces and we can then use Proposition 
\ref{ample}.

Let $X$ be any subspace of $V$ satisfying $\dim X+\dim 
\zeta(X)\ge 2m$ and containing an abelian subspace of dimension $m$. 
Let $B$ be the isolated subgroup $A$ corresponding to subspace $X$. 
Then 
$\rank B +\rank \zeta(B)\ge 2m$. Thus we can apply Lemma 
\ref{largecentre} to show that equality holds and that 
$F_1(B/\zeta(B))$ has a \hol\  module $M(B)$ (of dimension 
$m-\rank \zeta(B)$). Let $\Omega'_X$ denote the complete inverse 
image under the map $B \rightarrow B/\zeta(B)$ of the set of carrier 
space subgroups of $\Delta^*(M(B))$. These will have rank equal to 
$\dim M(B)+\rank \zeta(B)=m$. Further we may suppose $m\ge 2$, and 
if we apply Lemma \ref{nontrivinter} to $B$, we see that for each element $C$ of 
$\Omega'_X$, there is another element $C'$ so that $C+C'$ is of rank 
$2m$. Let  $\Omega_X$ denote the subspaces spanned by the elements of 
$\Omega'_X$. Then each element of $\Omega_X$ is supplemented by 
another element of $\Omega_X$ and, considering dimensions, we see 
that these two subspaces have trivial intersection. Thus condition 
(1) is satisfied in the requirement for ample abelian subspaces. 
Condition (2) follows immediately from Lemma \ref{trivinter}.

Thus 
$V$ has ample abelian subspaces and we can apply Proposition 
\ref{ample} to show that $V$ has a symplectic base. Note that, 
because we assumed the centre of $A$ to be trivial, then $\zeta(V)$ 
will also be trivial and we will have 
$$V=V_1\oplus \dots \oplus V_t.$$
Let 
$A_i$ be the isolated subgroup in $A$ corresponding to $V_i$. 
Then the $A_i$ will 
generate their direct product and this will have finite index in $A$. 
The remaining properties of the $A_i$ follow immediately.

It remains 
to prove that finitely generated $FA_i$-modules of $M$ are \hol. 
 From Proposition 2.5 of \cite{jrjgcjbb2}, $M$ has a critical 
submodule $M_0$. Let $N$ be a cyclic critical $FA_i$-submodule of $M$ 
having minimal dimension. We claim that $N$ is a \hol\  module for 
$FA_i$.

Firstly, note that $FA_i$ has centre $F$. Also observe that 
$N$ is torsion-free as $FB$-submodule for any commutative $FB$, since 
a similar statement is true for $M$. Denote by $2m_i$ the rank of 
$A_i$. It remains to show that $N$ has dimension $m_i$.

We can apply 
Theorem 5.5 of \cite{jrjgcjbb2} to show that 
$\pi_i(\Delta^*(M))^*=\Delta^*(N)$ where $\pi_i^*$ is the map $A^* 
\longrightarrow A_i^*$ induced by the injection $A_i \longrightarrow 
A$. Thus each carrier space of $\Delta^*(N)$ will be the image under 
$\pi_i^*$ of a carrier space of  $\Delta^*(M)$. We prove that, if $U$ 
is any carrier space of $\Delta^*(M)$, then $\pi_i(U)$ has dimension 
$m_i$ and it follows immediately that $\pi_i^*(N)$ has dimension 
$m_i$ and so that $N$ has dimension $m_i$.

Let $B_0$ denote the 
carrier space subgroup of $A$ which is dual to $U$. Thus $B$ has rank 
$m$ and, by (2) of Proposition \ref{cjbb}, $B$ has a subgroup of 
finite index $B_1$ so that $FB$ is commutative. But this implies, by 
Proposition \ref{absubsp}, that $B_1$ has a subgroup of finite index 
of the form $\oplus (B_1\cap A_i)$.  Since $B_1$ has rank 
$m=m_1+\dots+m_t$ and each $B_1\cap A_j$ has rank at most $m_j$, it 
follows that each $B_1\cap A_j$ has rank exactly $m_j$. In 
particular, $B_1\cap A_i$ has rank $m_i$. Thus $A_i/(B\cap A_i)$, and 
so also $(B+A_i)/B$ has rank $m_i$. Passing back to the dual, we see 
that $U/U \cap A_i^{\circ}$ has dimension $m_i$, where $A_i^{\circ}$ 
is dual to $A_i$; that is, it is the kernel of $\pi_i^*$.  Hence 
$\pi_i^*(U)$ has dimension $m_i$ as required and the proof is 
complete.

\section{Finitely presented 
abelian-by-nilpotent-of-class-two groups}

We now aim to convert the results of the previous section into 
results about finitely presented abelian-by-nilpotent-of-class-two 
groups. Suppose that 
$$\{1\} \longrightarrow M  \longrightarrow G 
\longrightarrow H  \longrightarrow \{1\}$$
with $M$ abelian, $H$ 
nilpotent of class 2 and with $G$ finitely presented. Thus $M$ is a 
finitely generated $\mathbb ZG$-module. We shall write the operation 
of $M$ as addition. 

We shall use the assumption of finite 
presentation for $G$ in the form of the somewhat weaker consequence 
that the second homology \\$H_2(M,Z)= M\wedge_{\mathbb Z} M$ is 
finitely generated when considered as a $\mathbb ZG$-module via the 
diagonal action. There is an immediate problem in  that we would like 
to pass to submodules $M_1$ of $M$ but that the natural map 
$M_1\wedge M_1 \rightarrow M\wedge M$ is, in general, not an 
injection. It is, however, an injection, in case either $M$ and $M_1$ 
are torsion-free or $M$ and $M_1$ are both of the same prime exponent 
$p$. In the former case we use the fact that $M$ and $M'$ are both 
flat $\mathbb Z$-modules and in the latter case, we observe that the 
exterior square over $\mathbb  Z$ is equal to the exterior square 
over $\mathbb Z/p\mathbb  Z$. Combining this with the fact that 
epimorphisms of modules yield epimorphisms of exterior squares, we 
obtain the following.

\begin{lemma} \label{wedge}Suppose that $M_1$ 
is a submodule of $M$ and that $M\wedge M$ is finitely generated as 
$\mathbb ZH$-module.
\begin{enumerate}
  \item if  $M_1$ is 
torsion-free, then $M_1\wedge M_1$ is  finitely generated as $\mathbb 
ZH$-module;
 \item if $M_1$ has prime exponent $p$ and the 
$p$-torsion subgroup of $M$ has bounded exponent then $M_1$ has a 
non-zero submodule $M_2$ so that $M_2\wedge M_2$ is  finitely 
generated as $\mathbb ZH$-module .
\end{enumerate}
\end{lemma}

\begin{proof} If $M_1$ is torsion-free then $M_1$ is isomorphic to a 
submodule of the quotient $\overline  M$ of $M$ by its torsion 
subgroup. Then $\overline M \wedge \overline M$ is finitely generated 
and contains a submodule isomorphic to $M_1\wedge M_1$.  Since $\mathbb 
ZH$ is Noetherian. it follows that $M_1\wedge M_1$ is finitely generated

If $M_1$ has prime exponent $p$ and if the $p$-torsion subgroup $T$ 
of $M$ has exponent $p^l$ then $p^lM\cap T=\{0\}$ and so $p^lM\cap 
M_1=\{0\}$. Choose $k$ minimal so that $p^kM \cap M_1 = \{0\}$. Then 
$M_1$ is isomorphic to a submodule of $M/p^kM$. Let $M_2$ be the 
complete inverse image in $M_1$ of  $\overline M=p^{k-1}M/p^kM$. The 
map $m \mapsto p^{k-1}m+p^kM$ is a homomorphism (of $\mathbb 
ZH$-modules) and so $\overline M\wedge \overline M$ is finitely 
generated. As $\overline M$ has exponent $p$ and has a submodule 
isomorphic to $M_2$, it follows that $M_2\wedge M_2$ is finitely 
generated.
\end{proof}

\begin{lemma} \label{half} Let $G$ be a finitely presented group with 
an abelian normal subgroup  $M$ with quotient which is
nilpotent of class 2.  Let $M_1$ be a non-trivial
$G$-normal subgroup of $M$ so that $M_1\wedge M_1$ is finitely 
generated as $\mathbb ZG$-module. Let $C$ denote the centraliser of
$M_1$ in $G$; denote $G/C$ by $H$ and let $Z$ denote the centre of 
$H$. Suppose:
\begin{enumerate}
  \item  $\Ann_{\mathbb ZZ} (M_1)=P$ 
is prime;
  \item $M_1$ is torsion-free as $\mathbb ZZ/P$-module;
 
\item  for some subgroup $K$ of $H$ with $Z \subseteq K$,  $M_1$ is 
not torsion as $\mathbb
ZK/(P.\mathbb ZK)$-module.
\end{enumerate}  
Then $$\rank (K/Z) \le 
\frac 12 \rank(H/Z).$$
\end{lemma}
\begin{proof}
As $M_1$ is not torsion as $\mathbb ZK/(P.\mathbb ZK)$-module, there 
is $m\in M_1$ so that, if we set
$V=m.\mathbb ZK$, then $V$ is isomorphic to $\mathbb ZK/(P.\mathbb 
ZK)$. Set $M_2=V.\mathbb ZH$. We wish to apply Lemma 9 of Segal 
\cite{ Seg77}
to $V$ and $M_2$. Set $J=\mathbb ZZ/P$. By assumption, $M_1$ and so 
$M_2$ is torsion-free as $J$-module. Then Lemma 9 of \cite{ Seg77} 
tells us
that there is a non-zero ideal $\Lambda$ of $J$ so that if $Q$ is any 
ideal of $J$ which does not contain $\Lambda$ then $M_2Q\cap
V=VQ$.

As $J$ is a finitely generated commutative domain, it follows from 
the Nullstellensatz that the Jacobson radical is trivial (see, for 
example, Section 4.5 of \cite{Eis}) and so
there is a maximal ideal $Q_1$ of $J$ which does not contain 
$\Lambda$; further $J/Q_1$ will be finite. Let $Q$ be the (maximal)
ideal of $\mathbb ZZ$ so that $Q/P=Q_1$. Then $M_2Q\cap V=VQ$ and so 
$M_2/M_2Q$ contains a copy of $V/VQ\cong \mathbb ZK/Q.\mathbb
ZK$.

By Theorem G of Segal\cite{ Seg77}, $M_2/M_2Q$ has Krull dimension at 
least that of the $\mathbb ZK$-module $V/VQ$. But $V/VQ$ is easily 
seen to be
a crossed product of the (central) field $\mathbb ZZ/Q$ with the 
group $Z/K$ and a minor adaptation of the proof of Smith\cite{ Smi} 
for
group rings shows that the Krull dimension of $V/VQ$ equals $\rank(K/Z)$. Thus
\begin{equation}\label{krullin1}
\kdim(M_2/M_2Q)\ge \rank(K/Z)
\end{equation}
  where $\kdim(-)$ denotes Krull dimension and $\rank(-)$ denotes the 
torsion-free rank or Hirsch length.

It follows from 2. that $M_1$ is either $\mathbb Z$-torsion-free or of finite prime exponent. Thus, because $M_1\wedge M_1$ is finitely generated as $\mathbb ZH$-module, so also is $M_2\wedge M_2$ and hence also $M_2/M_2Q\wedge M_2/M_2Q$.  Let $B$ denotes the centraliser of $M_2/M_2Q$ in $H$. 
As  $\mathbb ZZ/Q$ is a finite integral domain, it will have prime 
exponent as abelian group and so,  we can apply Lemma 3 of
Brookes\cite{Brex} to obtain that
\begin{equation} \label{krullin2}
2\kdim (M_2/M_2Q) \le \rank(H/B).
\end{equation}
As $\mathbb ZZ/Q$ is finite, $B\cap Z$ must have finite index in $Z$
and so $\rank(H/B)=\rank(H/BZ) \le\rank(H/Z)$. Hence, combining this 
with (\ref{krullin1}) and (\ref{krullin2}), we obtain that
$$ 2 h(K/Z)\le \rank(H/Z)$$ as required.
\end{proof}

We now aim to translate the results of the previous section into a 
result for groups. We begin with a special case.
\begin{prop} \label{submod}Suppose the the sequence of groups
$$ 
\{1\} \longrightarrow M \longrightarrow G \longrightarrow H 
\longrightarrow \{1\}$$
is exact with $M$ abelian, $H$ torsion-free 
nilpotent of class 2 and $G$ finitely presented. Suppose that $M_1$ 
is a non-zero $G$-normal subgroup of $M$ with $M_1\wedge M_1$ 
finitely generated as $G$-module. Then there exists a subgroup $G_0$ 
of finite index in $G$ and a non-zero $G_0$-normal subgroup $N$ of 
$M_1$ so that  the quotient of $G$ by the centraliser of $N$  is a 
central product of groups which are either cyclic or Heisenberg.
\end{prop}

\begin{proof} Let $Z$ denote the centre of $H$ and consider $M_1$ as 
$\mathbb ZZ$-module. Let $P$ be a maximal associated prime of $M_1$ 
and let
$N_1$ be the victim of $P$ in $M_1$; that is, $N_1=\{n\in M_1: nP=0\}$. 
Then $N_1$ is a $\mathbb ZH$-submodule of $M_1$ and is torsion-free
as $\mathbb ZZ/P$-module. As abelian group, $N_1$ must either be 
torsion-free or of prime exponent $p$ and so we can apply Lemma 
\ref{wedge} to show that $N_1$ has a non-zero submodule $N$ with 
$N\wedge N$ finitely generated as $H$-module. It now follows from 
Proposition 2 of \cite{BRsing} that $N$ is an impervious $\mathbb 
ZH$-module.

Let $F$ be the field of fractions of $\mathbb ZZ/P$.  As $N$ is 
torsion-free as $\mathbb ZZ/P$-module, it will embed into $\widehat
N= N\otimes_{\mathbb ZZ/P} F$ and the latter has a natural structure 
as $\mathbb ZH/(P.\mathbb ZH) \otimes_{\mathbb ZZ/P}
F$-module. But it is easy to check that this ring has a natural 
structure as a crossed product of the central field $F$ by the
free abelian group of finite rank $A=H/Z$. Thus we now have an 
$FA$-module $\widehat N$.

If $\widehat N$ has dimension $d$ then there is a subgroup $B$ of $A$ 
of rank $d$ so that $\widehat N$ is not torsion as
$FB$-module. If $K$ is the subgroup of $H$ so that $K/Z=B$, then it 
is easily verified that $N$ is not torsion as $\mathbb
ZK/P.\mathbb ZK$-module. But then, by Lemma \ref{half}, $d=\rank(K/Z) 
\le (1/2)\rank(H/Z)=\rank(A)$; that is
$$\dim(\widehat M)\le  \frac 12\rank(A).$$

It is straightforward to check (see for example, the proof of Theorem 
3.1 of [BG4]) that $\widehat N$ is also impervious as $FA$-module. It
is also clear that since $F$ is constructed from the centre of $H$, 
it will be the centre of $FA$. Then, by Corollary 5 of
Brookes\cite{Brcp}, $\dim(\widehat N)\ge  \frac 12\rank(A)$ and so 
$\dim(\widehat N)=\frac 12\rank(A)$. We also have, from Theorem 4 of
Brookes\cite{Brcp}, that $\widehat N$ is $FB$-torsion-free for any 
$F$-abelian subgroup $B$ of $A$. Thus $\widehat N$ is a \hol\  
$FA$-module. It now follows from Theorem \ref{main} that $A$ has a 
subgroup of finite index of the form $\oplus A_i$ where each $A_i$ 
generates a Heisenberg group within $FA$ and these Heisenberg groups 
commute.

Let $C$ denote the centraliser of $N$ within $H$. Then 
there exists subgroups $H_i$ of $H$  so that $H_i/C$ is Heisenberg, 
the subgroups $H_i/C$ are pairwise commuting  and so that the 
subgroups $H_iZ/C$  generate a subgroup of finite index in $H$. Thus 
$\prod H_i Z/C$ is of finite index in $H/C$ and so $H/C$ has a 
subgroup of finite index which is a central product of subgroups 
which are either Heisenberg or cyclic.
\end{proof}

Recall that a 
submodule is {\it essential} if it is non-zero and has non-zero 
intersection with every non-zero submodule.
\begin{lemma} 
\label{artrees}Let $H$ be a finitely generated nilpotent group and 
let $M$ be a finitely generated $\mathbb Z H$-module. If $N$ is an 
essential submodule of $M$ and $K$ is a normal subgroup of $H$ which 
acts nilpotently on $N$, then $K$ acts nilpotently on 
$M$.
\end{lemma}

\begin{proof} Let $J$ denote the kernel of the 
natural map $\mathbb ZH \rightarrow \mathbb Z(H/K)$. Because $K$ acts 
nilpotently on $N$, we have $NJ^m=\{0\}$ for some $m$. Set $I=J^m$. 
Because $I$ is an ideal of the group ring of a nilpotent group, it 
satisfies the weak Artin-Rees condition; that is, there is an integer 
$n$ so that $MI^n \cap N =NI$ (see Theorem 11.3.11 and Theorem 11.2.8 
of Passman\cite{Passcr}). But $NI=\{0\}$ and $N$ is essential. Thus 
$MI^n=MJ^{mn}=\{0\}$ and so $K$ acts nilpotently on $M$.
\end{proof}

\begin{thm}\label{maingp} Let $G$ be a finitely presented group with 
a normal abelian subgroup $M$ so that $H=G/M$ is torsion-free 
nilpotent of class 2. Then there exist normal subgroups $M_i$ of $G$, 
lying in $M$, and normal subgroups $C_i$ of $H$ so 
that
\begin{enumerate}
 \item  $H/C_i$ is virtually a central product 
of Heisenberg and cyclic groups;
  \item $C_i$ acts nilpotently on 
$M_i$;
  \item if $C=\cap _i C_i$ then $C$ acts nilpotently on 
$M$.
\end{enumerate}
\end{thm}

\begin{proof}We shall consider $M$ as 
$\mathbb ZH$-module. Since $\mathbb ZH$ is Noetherian, we can find a 
maximal finite direct sum of non-zero submodules $M_i$ of $M$ so that 
each $M_i$ is uniform. Then the sum $M'$ of the $M_i$ is necessarily 
essential.

Each $M_i$ necessarily has a non-zero submodule which is 
either of prime exponent or is torsion-free and, by Lemma 
\ref{wedge}, therefore has a non-zero submodule $N'_i$ so that 
$N'_i\wedge N'_i$ is finitely generated as $\mathbb ZG$-submodule. 
Thus, by Proposition \ref{submod}, there is a subgroup $H_i$ of 
finite index in $H$ and a non-zero $\mathbb ZH$-submodule $N_i$
of 
$N'_i$ so that if $C_i$ denotes the centraliser of $N_i$ then $H/C_i$ 
is a central product of groups which are either cyclic or 
Heisenberg.

But $N_i$ is a non-zero submodule of the uniform module 
$M_i$ and so is essential. Thus, by Lemma \ref{artrees}, $C_i$ acts 
nilpotently on $M_i$. Thus, if $C$ is the intersection of the $C_i$, 
then $C$ acts nilpotently on the sum $M'$ of the $M_i$. But $M'$ is 
an essential submodule of $M$ and so $C$ acts nilpotently on 
$M$.

\end{proof}

\begin{coroll} Let $G$ be a finitely presented 
groups which is an extension of an abelian normal subgroup by a group 
which is torsion-free nilpotent of class 2. Suppose that any two 
non-trivial normal subgroups of $G$ have non-trivial intersection. 
Let $F$ denote the Fitting subgroup of $G$. Then $G/F$ has a subgroup 
of finite index which is a central product of groups which are cyclic 
or Heisenberg.

\end{coroll}
\begin{proof} This is simply the case of  Theorem 
\ref{maingp} in which one of the subgroups $M_i$ can be taken equal 
to $M$.
\end{proof}


\providecommand{\bysame}{\leavevmode\hbox 
to3em{\hrulefill}\thinspace}
\providecommand{\MR}{\relax\ifhmode\unskip\space\fi MR }
\providecommand{\MRhref}[2]{%
 
\href{http://www.ams.org/mathscinet-getitem?mr=#1}{#2}
}
\providecommand{\href}[2]{#2}

\end{document}